\documentclass[11pt]{article}
\usepackage{amsmath,amssymb,amsfonts,latexsym}
\usepackage{amssymb}
\usepackage{amsthm}
\usepackage{newlfont}
\usepackage{graphicx}
\usepackage{subfigure}
\usepackage{float}
\makeatletter
\def \fixequation{\let \c@equation\c@theorem\let \p@equation\p@theorem\let \theequation \thetheorem}
\makeatother \theoremstyle{plain}
\newtheorem{theorem}{Theorem}
\newtheorem{lemma}{Lemma}

\newtheorem{claim}{Claim}
\theoremstyle{definition}

\theoremstyle{remark}

\begin{document}

\def\s{\sigma}
\def\A{\mathcal{A}}
\def\B{\mathcal{B}}
\def\ga{\gamma}
\def\La{\Lambda}
\def\la{\lambda}
\def\p{\phi}
\def\ep{\epsilon}
\def\sg{\sigma}
\def\al{\alpha}
\def\gp{\frak P}
\def\gq{\frak q}
\def\gz{\frak Z}
\def\gs{\frak S}
\def\gg{\frak g}
\def\kin{k \in [1,n]}
\def\kim{k \in [1,m]}
\def\ss{s^k_1, s^k_2}
\def\Jin{J \subset [1,n]}
\def\ts{\tilde{s}}
\def\mN{\mathbb{N}^+}
\def\mz{\mathbb{Z}}
\def\mp{\mathbb{P}}
\def\mzt{\mathbb{Z}_2}
\def\Cb{\mathbf{C}}
\def\bE{\mathbf{E}}
\def\qs{q_{\mbox{SAT}}}
\def\sbe{\mbox{SAT}(\mathbf{E})}
\def\rx{R[X]}
\def\rxo{R[X_0]}
\def\on{\otimes^N \mathbb{C}^n}
\def\onn{\otimes^n \mathbb{C}^n}
\def\one{\otimes^N \mathbb{C}^n(\eta)}
\def\noo{(v_{i_1} \otimes \cdots \otimes v_{i_N})}
\def\GL{GL_n(\mathbb{C}}
\def\sn{S_N(U_n)}
\def\mbc{\mathbb{C}}
\begin{center} {\bf \Large \  diagonals of real symmetric matrices
of given spectra as a measure space}\\
 Avital Frumkin and  Assaf Goldberger

\end{center}
\subsection*{\large \bf Abstract }
The set of diagonals of real  symmetric matrices of given  non negative  spectrum is endowed with a measure which is obtained  by the push forward of the Haar measure of the  real orthogonal group.\\
We prove that the Radon Nicodym derivation of this measure with respect to the relative Euclidean measure   is  approximated  by the coefficients of a sequence of zonal sphere polynomials corresponding with the given spectrum.
 There is a striking similarity between the role of the  zonal sphere polynomials in  the orthogonal case, and that of the  Schur function  in the Hermitian case.\\
Following  this  we obtain a combinatorial approximation for the probability of  real symmetric matrix of a given spectrum to appear as the sum of two real symmetric matrices, each of a given spectrum. In addition  we obtain  a   real orthogonal  analogue to the Zuber Itzykson Harish Chandra integration formula.

\subsection*{\large \bf Preliminaries and main result }
    We use Greek  letters to denote  partitions of unity, \\ $\delta= \delta_1,\delta_2...\delta_n $ where $ \delta_i\geq 0$ and $\sum_i{\delta_i} =1$, as well as for partitions of integers,  $\gamma\vdash N  $.

 Let $\lambda=(\lambda_1,\lambda_2,....,\lambda_n)$ be a positive
real
vector so that $\sum \lambda_i =1$\\

 For simplicity, we
assume that all  $\lambda_i$ are rational.\\
 Define a diagonal matrix by
 \begin{equation}
D_\lambda =diag(\lambda_1,\lambda_2,...,\lambda_n)
\end {equation}
\\ Let $O_n$ be
the real orthogonal group of n by n matrices.\\
By the Horn-Schur theorem, the set \begin{equation}( {{(oD_\lambda
o^t})_{11},{(oD_\lambda o^t})_{22}, ...,{(oD_\lambda o^t})_{nn} / o\in O_n})\end {equation}
 is the
permutahedron $PH_\lambda$, defined by the vector $\lambda$. That is the convex hull
of all the vectors
$(\lambda_{\sigma(1)},\lambda_{\sigma(2)}...\lambda_{\sigma (n)})$
  for $\sigma\in S_n$. \\
  $PH_\lambda $ is endowed with a
 measure, given by pushing forward of the Haar
measure of the real orthogonal group. We denote this measure  by
$DH_\lambda^O$,
(In the Hermitian case $DH_\lambda^U$(see [FG]))\\
 We asumme that the Radon Nicodym derivation of $DH_\lambda^o$ with respect to the relative Euclidian measure of $PH_\lambda$ exists, and is  almost everywher contiuous. Furthermor
 the suspected points of the Radom Nicodym derivation function are included in severel afine spaces, their intersections with $PH_\lambda$, is of posetiv codimension.
  The proof will be given elswhere.\\

  {\bf Remark} As not as in our case,  the Radon Nicodym derivation of $DH_\lambda^U$ with respect to the relative Euclidian measure of $PH_\lambda$ can be  computed directly using representation theory, thereby  avoiding  any need for  differential geometry. See[F G] for more details.\\
    Let $\lambda$ be a partition of a positive integer  N ($\lambda\vdash N$). Further,let\\

 \begin{equation}Z_{\lambda}(X)=\sum_{\eta\vdash N} a_\eta X^\eta\end {equation}
 denote   the zonal sphere polynomial
  corresponding to the partition $\lambda$ (see [Mac] section 7 and [J 1])\\
   The next thorem  is our main result.
   \begin{theorem}
Let $\lambda=\lambda_1\leq \lambda_2...\leq \lambda_n$ be a non negative
rational vector so that $ \sum \lambda_i =1$, and let  N be a positive integer, such that
$N\lambda$ is a vector of integers.\\ Let
$Z_{N\lambda}(X)=\sum_{\eta\vdash N} a_\eta^N X^\eta$ be  the
corresponding zonal sphere polynomial. Assume  $\frac{\eta}{N}$ is a continuous point of the density function (Radon Nicodym derivation of $DH_\lambda^o$ with respact to the relativ Euclidian measure on $PH_\lambda$). Than as N tends to
infinity, $\frac{a_\eta^N}{Z_{N\lambda}(I_n)} $ ($I_n$ is the corresponding unit matrixe) approximates this density function   at
the point $\frac{\eta}{N}$ in $ PH_\lambda$.
\end{theorem}
{\bf Remark} Compare to the characterization of the zonal spher polynomials coefficient in [Mac 2.27]page 409

  {\bf Remark} In the Hermitian  case
the zonal sphere polynomials  are replaced by  the Schur functions
in order to obtain  a similar approximation. See[F G] for more details  \\
  \subsection*{\large \bf  Corollaries of the main result}

   Theorem 1 leads to two corollaries we introduce now
  \\
  Let $\gamma= (\gamma_1,\gamma_2,...,\gamma_n)$ be another   vector of rational non negative  numbers   such that
  $\sum \gamma_i =1$, and $D_\gamma=diag(\gamma_1,\gamma_2,...,\gamma_n)$.\\
 Let N be a large integer so that $N\lambda$ and $N\gamma$ are
integral vectors . Let
\begin{equation}
Z_{N\lambda}(X)Z_{N\gamma}(X)=\sum_{\eta\vdash 2N} {C_\eta^{N\lambda,
N\gamma}}Z_\eta( X)\end{equation}
be the
linearization of the product of the two zonal sphere polynomials.(see [Mac] page 409 )\\
\begin{claim}
Using the Haar measure of the orthogonal group as a probability
measure, the Radon-Nikodym derivative at $\frac{\eta}{N}$ of the distribution of the random variable,$\frac{\eta'}{N}$ satisfying the equation
$\frac{oD_\lambda o^t+o'D_\gamma o'^{t}}{2}=D_{\frac{\eta'}{N}}:o,o'\in O_n $, with respect to the relative Euclidian  measure on the simplex, is
 approximated, up to normalization, by
$C_\eta^{N\lambda, N\gamma}$.\\
\end{claim}
Simply speaking, the probability  to obtain a matrix with spectra close to $\frac {\eta}{N}; \eta\vdash N$, as the sum of two symmetric matrices of spectra $\lambda$ and $\gamma$, chosen randomly with respect to the Haar measure of the real orthogonal group,  can be evaluated  by the  linearization  coefficients as above \\

 The reason for this is that one can define a measure  on the set \\$\{\frac{1}{2}((D_\lambda^o+D_\gamma^{o'})_{11},(D_\lambda^o+D_\gamma^{o'})_{22},...,(D_\lambda^o+D_\gamma^{o'})_{nn})\}$ \\(using  the notation  $A^o\equiv oAo^t$ ), given by the push forward the Haar measure of $O_n \times O_n$. Since this measure  is $DH_\lambda^o*DH_\gamma^o$ (the convolution of the two measures),the main result tells us that it can be approximated  by the coefficients of the product $Z_{N\lambda}(X)Z_{N\gamma}(X)$\\
{\bf Remark} In the Hermitian case, the
Littlewood Richardson coefficients
 replace the coefficients  $C_\eta^{N\lambda,N\gamma}$ in (4). See [F G] for more details\\
The second corollary   deals with the real  symmetric version of the
Zuber Itzykson Harish Chandra integral.\\
  Assume that C is a real
symmetric matrix. For  $\lambda$ and $D_\lambda$ as in (1), define
\begin{equation}I_\lambda (C)\equiv\int exp (trace(CD_\lambda^o))dO
\end{equation}
integrated over $O_n$ with respect to its Haar measure.\\
 \\ Without loss  of generality, one can assume that
C is a diagonal matrix. Now, the integral may be  taken over $PH_\lambda$,
with
respect to $DH^o_\lambda$. \\
 $C_1,C_2...C_n$ be the spectra of C, then for arbitrary N we have
 \begin{equation}I_\lambda (C)=\int_{PH_\lambda}\prod (exp \frac{C_i}{N})^{N{(D^o_\lambda)_{ii}}} d(DH_\lambda^O)\end{equation}  By  using Riemann-sums
 to approximate the integration over $PH_\lambda$, computing  the integrand at the points of type
 $\frac{\gamma}{N}:\gamma \vdash N$ in  $PH_\lambda$,  and then by the main result approximate  the density function of
 $DH_\lambda^o$ at those points using  the coefficients of the zonal sphere polynomial $Z_{N\lambda}(X)$  we obtain the following.\\

\begin{claim} If $\lambda $ is a unit partition as given above,then
 \begin{equation}
  I_\lambda (C)\sim \frac{1}{Z_{N \lambda}(I_n)}Z_{N \lambda}(exp\frac
 {C_1}{N},exp\frac
 {C_2}{N},...exp \frac {C_n}{N})\end{equation}
 \end{claim}
 {\bf Remark} One can ignore the points of $PH_\lambda$ in which the Radon Nicodym derivation is not continues thanks to our assumptions on $DH_\lambda$.\\
 {\bf Remark} In the Hermitian version of this process, one may replace  the zonal
 sphere
 polynomial in the right hand side  with the Schur function with respect to the same parametrization.
 The very elegant  determinant formula of  Harish Chandra Zuber Itzykson is accepted as  the limit
    of the    determinants formula for the corresponding  sequence of  Schur functions. Unfortunately,  the zonal sphere polynomial, do not yield such nice expression for their computations  \\
\subsection*{\large \bf Proof of the main result}

\begin{proof}
Let us define a sequence of polynomials as follows:
\begin{equation} P_\lambda^N(X)=\int_{O_n} trace^N (D^o_\lambda D_X)do \end{equation} once integrated over
the orthogonal group with
respect to its Haar measure,\\
and X is a vector of indeterminants  $X_1,X_2,...X_n$.\\
{\bf Remark} In [J 2] Identity\\
\begin{equation}P_\lambda^N(X)=\sum_{\eta \vdash N}
c_\eta Z_\eta(D_X) Z_\eta(D_\lambda) \end{equation} was obtained, and the coefficients $c_\eta$ were investigated .However  we avoid it in our treatment  .\\ Now we rewrite
(8)as
\begin{equation} trace^N(D^o_\lambda D_X) =(\sum (D^o_\lambda)
_{ii}X_i)^N\end{equation}
 One can observe  the similarity  between the sequence of polynomials,
$P_\lambda^N(X)$, and  the sequence of  Bernstein polynomials
approximating the density of $ DH^o_\lambda$ (see [Lo] page 54)\\
More explicitly, by writing
\begin{equation}
P_\lambda^N(X)=\sum_{\eta \vdash N}a_\eta^N X^\eta
\end{equation}
and combining with (10), we obtain
\begin{equation}
a_\eta^N=\int  (^N_\eta)\prod_i{({D_\lambda}^o)_{ii}}^{\eta_i}dO
\end{equation}
Where the integral is  over the orthogonal group with respect to its Haar
measure.\\
Let $Y=Y_1,Y_2,...Y_n$ be a real vector   such that $\sum_i Y_i=1 : Y_i\geq 0$.\\
The next polynomials
\begin{equation}
  (^N_\eta)\prod_i{Y_i}^{\eta_i}
\end{equation}
are  special members  of the family of  "Dirichlet distributions", so the integrand of equation (12)
  tends to  $\delta_{\frac{\eta}{N}}$
 (the delta function concentrated
  around the points $\frac{\eta}{N}$) over the n-1 dimensional simplex, as N tends to infinity\\

  Therefore, as N tends to infinity, one finds that  the coefficients  $a_\eta^N$ tend   to the density of
  $DH_\lambda^o$ at the points $\frac{\eta}{N}$ in which it  is continuous.\\Next,
  we treat the left hand side of (10) by
writing
\begin{equation}
trace^N (D^o_\lambda D_X)=trace((D^o_\lambda D_X)^{\otimes
N})\end{equation} Let V denotes an n dimensional vector space over
the reals, and let $v_1 ,v_2,..v_n$ be an orthogonal basis of V.\\
Let us define the $\eta$ weight space for any partition $\eta$ of N, of no more than n
parts, by the formula
\begin{equation}{V^{\otimes N}}_\eta= span \{v_{i_1}\otimes v_{i_2}\otimes...\otimes
v_{i_N}: \#\{l : i_l=j\}=\eta_j\} \end{equation}  Now because  $D_\lambda $ is a
diagonal matrix, we have

\begin{equation}trace^N (D_\lambda D_X)=trace((D_\lambda D_X)^{\otimes N})=\sum_\eta
trace(D_\lambda D_X)^{\otimes N})|_{{V^{\otimes
N}}_\eta}\end{equation}\\and
\begin{equation}trace((D_\lambda D_X)^{\otimes N})|_{{V^{\otimes N}}_\eta}= \left(
\begin{array}{l} N\\ \eta \end{array} \right)\prod \lambda_i^{\eta_i}X_i^{\eta_i}\end{equation}
Now, using  spectral factorization of real symmetric operators, one obtain that
\begin{equation}(D_\lambda)^{\otimes N}=\sum_{\eta\vdash N}\gamma_\eta\sum \alpha_\gamma
^t\alpha_\gamma \end{equation} where the second summation runs  over
a basis of the the eigenspace   of weight $\eta$ with eigenvalue $\gamma_\eta=\prod\lambda_i^{\eta_i}$, \\
(One can choose the basis  as the standard unit vectors in $V^{\otimes N}_\eta$).\\

Now,one can use the central limit law  (using equation (17))  to  write \begin{equation}
P_\lambda^N(X)\sim \int_{O_n}\sum_{\eta\vdash N}\gamma_\eta \sum_\gamma trace (
(o^{\otimes N} \alpha_\gamma )(o^{\otimes N}\alpha_\gamma)^t D_X^{\otimes N} )do
\end{equation}.
\\summed over $\eta \vdash N $ such that $\frac{\eta}{N}\sim
\lambda$.\\
The first $\sim$ in the  equation means  that
contribution of the other summands to the coefficients $a_\eta^N$  vanishes to zero.
  The second $\sim$  denotes the closeness  of real vectors at
any coordinate.
\\ To get closer to our goal  one may apply representation
theory of $Gl_n(R)$\\
 By the  Schur Weyl duality we have \\
\begin{equation}V^{\otimes N} =\bigoplus V_\eta \otimes M_\eta
\end{equation}where the sum is over $\eta\vdash N $ of no more than n parts,\\
 and $V_\eta(M_\eta)$ are the $Gl_n(R)(S_N)$ simple
modules in their actions on the tensor product space.See [G W] page 138  for more details\\The
striking fact is that, as  N tends to infinity
\begin{equation}
trace^N (D_\lambda D_X)\sim\sum_{\frac{\eta}{N}\sim \lambda} trace ((D_\lambda D_X)^{\otimes
N}|_{V_\eta \otimes M_\eta })
\end{equation}
     (see the remarks on$\sim$ after (19)).\\ The discovery of this  concentration phenomena is, some
  times, attributed  to
Kyel and Werner  and has been  rediscovered several  times. (See [Ch] for a nice proof of it).\\
{\bf Remark} At this point, one can easily see the
full Hermitian perspective    of the problem. See [F G] for more details . For  symmetric real matrices one has more work to do.\\

Using (19) and (21), one can write
\begin{equation} P_\lambda^N(X)\sim \int\sum_\gamma c_\gamma (trace ( (o^{\otimes N} \alpha_\gamma
)( o^{\otimes N}\alpha_\gamma)^t D_X^{\otimes N} ))do \end{equation} summed over a basis of eigenvectors (of eigenvalues $c_\gamma$) of $D_\lambda^{\otimes N}$ in
$ (V_{\delta'}\otimes M_{\delta'})\bigcap {(V^{\bigotimes N})}_\delta,$ where $\frac{\delta}{N} ,\frac{\delta'}{N}\sim \lambda$.\\

Now, for each vector $\alpha $ in $V^{\otimes N}$, we define a polynomial
by the formula
\begin{equation} Q_\alpha(m)=trace( \alpha^t\alpha m^{\otimes N})\end{equation}
for any n by n  matrix m\\ Before introducing  the next lemma, we quote
  the definition of zonal sphere polynomials  from the first page of  [J 2]\\
"... Denote by $V_f$ the space of polynomials of degree f over the
real
symmetric positive definite n by n matrices.\\ $GL_n(R)$ acts on$V_f$ as follows \\ (For $g\in GL_n(R),P\in V_f: g (P(m))=P(gmg^t)$).\\
As $GL_n(R)$ module $V_f=\bigoplus_{(f)} V_{(f)}$ where
$(f)=f_1,f_2...f_n$ is a partition of f and $V_{(f)}$ is the subspace of $V_f$
on which $GL_n(R)$ acts  irreducibly corresponding to the partition 2(f) \\
The zonal polynomial $Z_f$ spans   the  real orthogonal group
invariant  vector space in $V_{f}$... "\\To complete the proof of Theorem 1 we need the following lemma
\begin{lemma} Using the notations from [J2], for $\alpha_\gamma$ as in (22),we have that $Q_{\alpha_\gamma}(m) \in \bigoplus V_{(f)}$ summed over (f) such that $\frac{(f)}{N}\sim
\lambda$.

\end{lemma}

\begin{proof}
 One need to apply some basic representation theory of $gl_n(R)$, the Lie algebra
 of $Gl_n(R)$ . [G W] page 228 is good reference  to refer.
We note that the weights  referred  to in (15) and the weights
in the context of Lie algebras representation theory coincide for an appropriates choice of basic positive  roots vectors.

 Now, since
the map  $\alpha \mapsto \alpha^t\alpha$ is bilinear, it can
be factorized through the tensor product.  $\alpha\mapsto \alpha\bigotimes\alpha\mapsto \alpha^t\alpha$. \\
Furthermore  for $\delta\vdash N$ we have$V_\delta\bigotimes V_\delta =\bigoplus
V_\omega$, summed over $\omega \preceq 2\delta$, where $\preceq$ is
the dominant order defined on the weight lattice ( see page
237 in [G W])\\ Since $\alpha_\gamma$ is of weight $\delta'$,
$\alpha_\gamma \bigotimes\alpha_\gamma$ is of weight $2\delta'$, and
so belongs to $\bigoplus V_\omega$ summed on $\omega$ in the next small fragment:
$2N(\lambda-\epsilon)\preceq 2\delta'\preceq\omega\preceq
2N(\lambda+\epsilon')$,  where $\epsilon,\epsilon' $ are arbitrarily small positive vector
,  as N tends to infinity . The central inequality
just expressed the fact that the weights of vectors in a highest weight modules
are smaller than its highest weight.
\end{proof}
  Given Lemma 1, Theorem 1 is now proved since
  the integral in (22) is just a projection onto the space spanned
by the   unique zonal polynomial in each $V_(f)$.

\end{proof}
\subsection*{\large \bf Numerical Experiments:}

We  test the $3\times 3$ case for the diagonal matrix $diag(0,1,2)$, conjugated by (a) $SO(3)$, and (b) $SU(3)$ (Figure 1). In each case we sample  30,000 conjugations,
and compute the resulting diagonal. In both cases we have a two-dimensional distribution, based on the $(1,1)$ and $(2,2)$ positions of the diagonal.
The $(3,3)$ entry is determined from the first two by the trace condition. .\\
We have also  test  conjugations of rank 1 matrices \\
We  sample   the diagonals of a $3\times 3$  matrix of rank 1  conjugated by $SO(3)$ random matrices (See Fig 2 (a))
On the other hand we   sample the zonal sphere  polynomial's coefficients  corresponding to one row (using example 1 of[Mac]Page 410)(See Fig 2 (b))\\
In Figure 3 the coefficients of the zonal sphere polynomial of two variables  corresponding to one row (the red points), have been compared with the function; $ xy^{-\frac{1}{2}},x+y=constant$,\\ which is computed in [FG] explicitly for two by two symmetric matrices
for that case. In the figures, we use the notation $J^r_s$ as employed in [Mac, p. 385] for the Jack-polynomials. Note that the $r=2$ case corresponds to the zonal polynomials.

\begin{figure}[H]
\hfill
\subfigure[Conjugation by $SO(3)$]{\includegraphics[width=5cm]{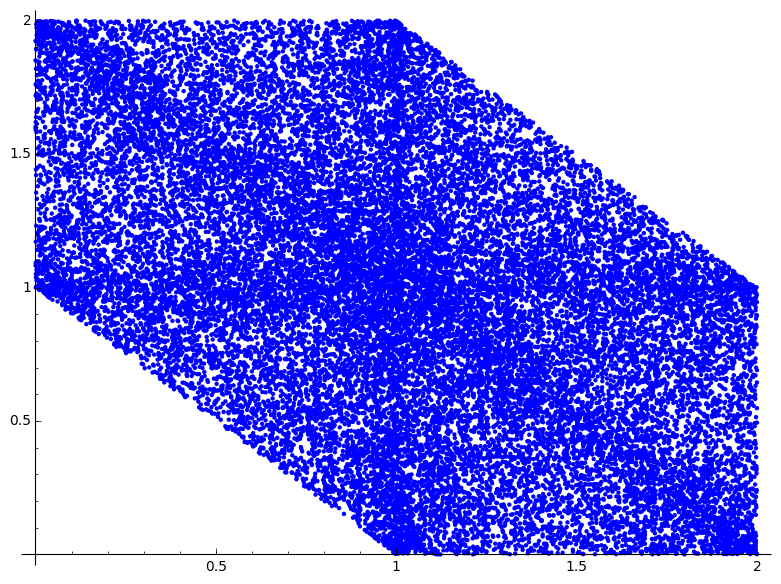}}
\hfill
\subfigure[Conjugation by $SU(3)$]{\includegraphics[width=5cm]{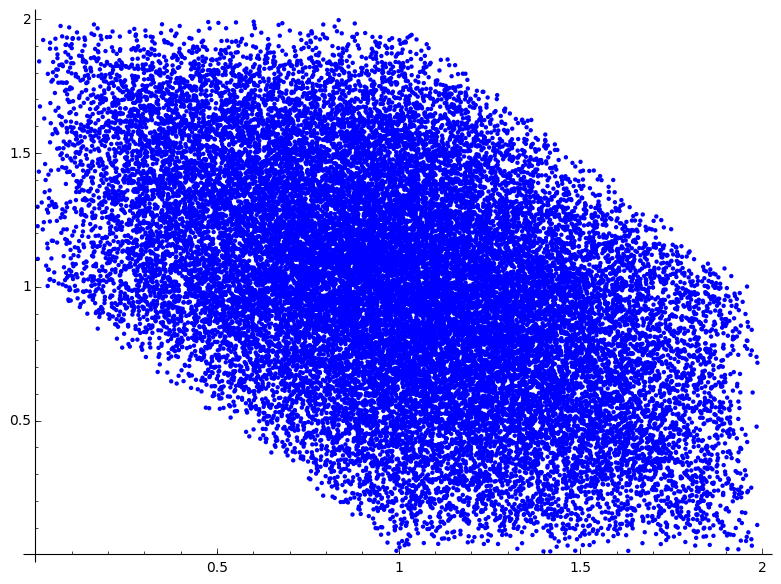}}
\hfill
\caption{}
\end{figure}

\begin{figure}[H]
\hfill
\subfigure[Rank 1: Conjugation by $SO(3)$]{\includegraphics[width=5cm]{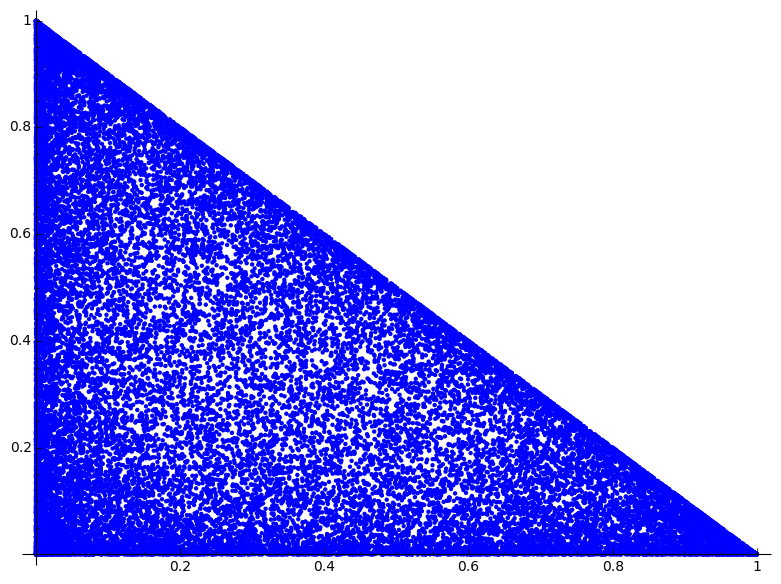}}
\hfill
\subfigure[Zonal Polynomial's coefficients  $J^2_{40}$]{\includegraphics[width=5cm]{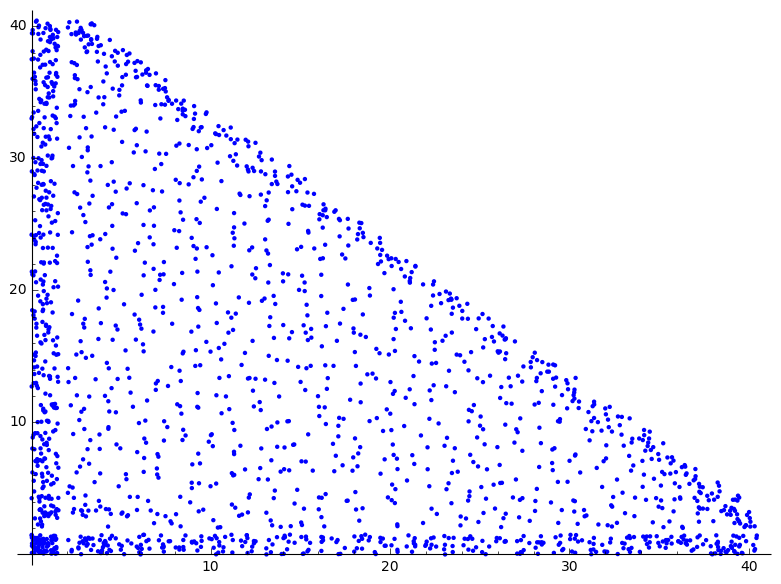}}
\hfill
\caption{}
\end{figure}

\begin{center}
\begin{figure}[H]
\hfill
\includegraphics[width=5cm]{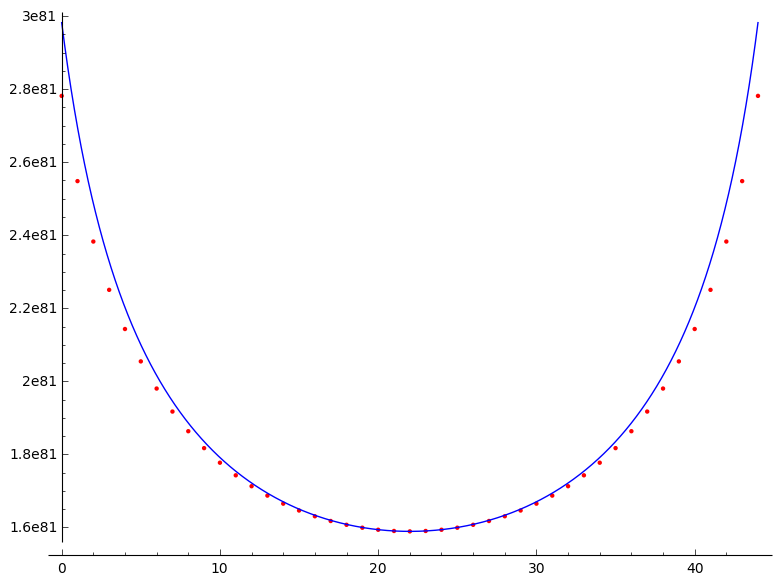}
\hfill
\hfill
\caption{Zonal coefficients of $J^2_{50}$(dotted) against  theoretical prediction}
\end{figure}
\end{center}

\section*{References}
\begin{itemize}
\item [{[J 1]}] Alan.T. James.  Zonal polynomials of the real positive
definite symmetric matrices Annals of math vol 74 no 3 November 1961\\
\item[{ [J 2]}] Alan.T.James. The distribution of the latent roots of the
 covariant matrix. Annals of math statist vol 31number 1(1960)\\
 \item [{[Mac]}] I,G.Macdonald.  Symmetric functions and Hall polynomials .
 Second edition Oxford mathematical monographs\\
\item[{ [G W]}] Roe Goodman and Nolan.R.Wallach. Representations and
invariants of the classical groups
 Cambridge U.press 1998
\item[{[Ch]}]M. Christandl, A. Harrow, G. Mitchison. Non-zero
Kronecker coefficients and consequences for spectra.  Communication
in Math Physics, 270(3) (2007), 575-585.
\item [{[Lo]}] G.G. Lorentz. Bernstein polynomials. Oxford clarendon
press(1979)
\item[{[FG]}] A.Frumkin, A.Goldberger.  On the distribution of the spectrum of the sum of two Hermitian or real symmetric matrices.
Advances in Appl. Math. 37 (2006), 268-286.
\end{itemize}

\end{document}